\documentclass[12pt,reqno]{amsart}
\usepackage{etoolbox}
\makeatletter
\patchcmd\maketitle
{\uppercasenonmath\shorttitle}
{}
{}{}
\patchcmd\maketitle
{\@nx\MakeUppercase{\the\toks@}}
{\the\toks@}
{}
{}{}
\patchcmd\@settitle{\uppercasenonmath\@title}{\Large}{}{}
\patchcmd\@setauthors
{\MakeUppercase{\authors}}
{\authors}
{}{}
\makeatother
\usepackage{amsmath, amsthm, amscd, amsfonts, amssymb, graphicx, color}
\usepackage{url}
\usepackage{tikz-cd}
\usepackage{hyphenat}
\usepackage[english]{babel}

\usepackage[utf8]{inputenc}
\usepackage[T1]{fontenc}
\newtheorem{theorem}{Theorem}[section]
\newtheorem{definition}{Definition}[section]

\newtheorem{corollary}{Corollary}[section]
\newtheorem{proposition}{Proposition}[section]
\newtheorem{lemma}{Lemma}[section]
\newtheorem{remark}{Remark}[section]
\newtheorem{example}{Example}[section]

\numberwithin{equation}{section}

\usepackage[colorlinks=true]{hyperref}
\hypersetup{urlcolor=blue, citecolor=red , linkcolor= blue}
\usepackage[capitalise,noabbrev,nameinlink]{cleveref}      

\def\bh{\mathcal{B}(\mathcal{H})}
\def\bp{\mathcal{B}_p(\mathcal{H})}
\newcommand{\bph}{\mathcal{B}_p(\mathcal{H})}

\def\h{\mathcal{H}}

\def\h{\mathcal{H}}
\def\R{\mathbb{R}}
\def\C{\mathbb{C}}

\def\obj{{\perp}^p}

\DeclareMathOperator{\tr}{tr}

\begin{document}
\author[T. Bottazzi and  C. Conde ] {\Large{Tamara Bottazzi}$^{1_{a,b}}$ and  \Large{Cristian Conde}$^{1_b, 2}$}
\address{$^{[1_a]}$ Universidad Nacional de R\'io Negro. Centro Interdisciplinario de Telecomunicaciones, Electrónica, Computación Y Ciencia Aplicada (CITECCA), Sede Andina (8400) S.C. de Bariloche, Argentina.}
\address{$^{[1_b]}$ Consejo Nacional de Investigaciones Cient\'ificas y T\'ecnicas, (1425) Buenos Aires,
		Argentina.}
\email{\url{tbottazzi@unrn.edu.ar}}
\address{$^{[2]}$ Instituto de Ciencias,Universidad Nacional de General Sarmiento, Los Polvorines, Argentina.}
\email{\url{cconde@campus.ungs.edu.ar}}

\date{\today}

\title[ Semi-inner product and  angles in Schatten ideals]
{ Semi-inner product and  angles in Schatten ideals}
\maketitle

\begin{abstract}
In this paper, we investigate the Schatten $p$-class ideals $\bp$ for $p > 1$ as semi-inner product spaces in the sense of Giles and Lumer. Within this framework, we explore several geometric and analytic notions such as Birkhoff-James orthogonality, $p$-parallelism, and related properties that naturally arise when these structures are interpreted through the lens of the associated semi-inner product. Furthermore, we introduce a novel notion of angle adapted to this context, which generalizes and unifies existing angle definitions in normed spaces. Our results contribute to a deeper understanding of the geometry of $\bp$ and offer new perspectives on operator behavior in semi-inner product spaces.

\end{abstract}

\section{Introduction and preliminaries}

\textcolor{red}{}

Let $(\mathcal{X}, \|\cdot\|)$ be a normed space.  Within the expansive realm of normed spaces, there exists a distinguished subfamily characterized by intrinsic and geometric properties that are intricately tied to the concept of an inner product. These spaces, commonly known as inner product spaces, not only preserve the fundamental attributes of normed spaces but also provide a rich geometric framework for interpreting essential concepts such as length, angle, and orthogonality among vectors.

The introduction of an inner product in these spaces imparts additional structure to the norm, facilitating the definition of angular and projective relationships between vectors. This enhanced structure is indispensable in various fields, including functional analysis, geometry, and mathematical physics, where the interplay between geometric and algebraic properties of vectors is of paramount importance.

Among inner product spaces, the Hilbert spaces are perhaps the most notable. Their completeness with respect to the norm induced by the inner product makes them the ideal setting for the study of operators and problems in infinite dimensions.

	A semi-inner product space (SIPS) is a generalization of an inner product space, in which the defining properties of the inner product are relaxed to some extent. Semi-inner product spaces maintain many useful characteristics of inner product spaces, but they are more flexible and can apply to a wider variety of settings. Following the work of  Lumer \cite{lumer} and Giles \cite{giles}, we recall that in any normed space $(\mathcal{X}, \|\cdot\|)$ one can construct a semi-inner product. This is a mapping $[\cdot, \cdot]:\mathcal{X}\times \mathcal{X}\rightarrow\mathbb{K}$,  where $\mathbb{K}$ represents either the real field $\mathbb{R}$ or the complex field $\mathbb{C}.$ The semi-inner product satisfies the following for all vectors $x, y, z$ in a vector space $\mathcal{X}$ and for all scalars $\alpha \in \mathbb{K}$:

	\begin{itemize}
		\item[(1)] $[x, x] = \|x\|^2$,
		\item[(2)] $[\alpha x + \beta y, z] = \alpha[x, z] + \beta[y, z]$,
		\item[(3)] $[x, \gamma y] = \overline{\gamma}[x, y]$,
		\item[(4)] $|[x, y]|^2\leq \|x\|^2\|y\|^2$,
	\end{itemize}
	
	The main difference between semi-inner product spaces and inner product spaces lies in the linearity condition. In inner product spaces, the product is linear in both arguments, whereas in semi-inner product spaces, the product is only guaranteed to be linear in one of the arguments.
	
	This relaxation of the linearity condition makes semi-inner product spaces more general. For instance, they can be applied to more abstract spaces, such as certain Banach spaces, where a full inner product might not be available. However, because of this reduced structure, semi-inner product spaces may lack some of the powerful properties of inner product spaces, such as the ability to define a norm directly from the inner product in the form $
	\|x\| = \sqrt{\langle x, x \rangle}.$
SIPS can be found in more general contexts, such as certain optimization problems or non-smooth analysis, where a full inner product structure is too restrictive.

In summary, while inner product spaces offer a highly structured environment with both conjugate symmetry and linearity in both arguments, semi-inner product spaces relax these conditions, allowing for more general applications but at the cost of some of the structure and convenience of inner product spaces.
	
The manuscript is organized as follows. 	In Section~\ref{s2}, we present the definitions, notations, and preliminary concepts required for our development. In Section~\ref{s3}, we explore the ideal $\bp$ as a semi-inner product space for $1 < p < \infty$. Finally, in Section~\ref{s4}, we propose various notions  of angle within this space, aimed at recovering its underlying geometric structure, including Birkhoff–James orthogonality, parallelism, and other related concepts.

	\section{Basic notions and preliminaries}\label{s2}

Let $\mathcal{B}(\mathcal{H})$ denote the $C^*$-algebra of all bounded linear operators on a separable Hilbert space $(\mathcal{H}, \langle \cdot, \cdot \rangle)$. An operator $X \in \mathcal{B}(\mathcal{H})$ is said to be positive if
$
\langle Xz, z \rangle \geq 0
$
for all $z \in \mathcal{H}$. We denote by $\mathcal{B}(\mathcal{H})^+$ the set of all positive operators in  $\mathcal{B}(\mathcal{H})$; that is,
\[
\mathcal{B}(\mathcal{H})^+ = \{ X \in \mathcal{B}(\mathcal{H}) : \langle Xz, z \rangle \geq 0 \text{ for all } z \in \mathcal{H} \}.
\]
Recall that any positive operator $X$ allows the definition of a unique positive square root operator denoted by $X^{\frac 12}$.

 For each $X \in \mathcal{B}(\mathcal{H})$, we define $|X|$ as the square root of $X^*X$. The operator $X$ can then be represented in its polar decomposition as $X = U|X|$, where $U$ is a partial isometry. The polar decomposition satisfies the following equalities
\begin{equation}\label{prop_polar_decomp}
U^*X=|X|,\: U^*U|X|=|X|, \:U^*UX=X, \: X^* = |X|U^*\: \textrm{and}\: |X^*|=U|X|U^*.
\end{equation}

The Schatten $p$-ideals, commonly denoted by $\bp$ for $1 \leq p < \infty$, are specific types of operator ideals associated with bounded linear operators defined on  $\mathcal{H}$. They consist of compact operators whose singular values lie in the $\ell^p$-sequence space, and these ideals are central to operator theory and functional analysis due to their close relationship with $L^p$-spaces and trace-class properties.

More precisely, for an operator $X\in \mathcal{B}(\mathcal{H})$, let $\{s_n(X)\}_{n=1}^\infty$ represent the sequence of singular values of $T$ \textcolor{red}{X}, which are the eigenvalues of the positive operator $|X|$, ordered in descending magnitude  and counting multiplicities. Then, the Schatten $p$-class $\bp$ is defined as follows:

$$
\bp=\left\{ X \in \mathcal{B}(\mathcal{H}) : \| X \|_p = \left( \sum_{n=1}^\infty s_n(X)^p \right)^{\frac{1}{p}} < \infty \right\}.
$$

Here, $\| X\|_p$ denotes the $p$-norm of $X$, known as the Schatten $p$-norm. Each Schatten $p$-class $\bp$ forms a two-sided ideal within $\mathcal{B}(\mathcal{H})$, meaning that if $Y \in \mathcal{B}(\mathcal{H})$ and $X \in \bp$, then both $XY$ and $YX$ are in $\bp$ as well.

Notable special cases include:
 $\mathcal{B}_1(\mathcal{H})$ corresponds to the trace-class operators, where each operator has a well-defined trace, and 
 $\mathcal{B}_2(\h)$  corresponds to the Hilbert-Schmidt operators, with an inner product given by 
 \begin{equation}\label{innerprod in b2}
 \left\langle Y,X\right\rangle_2=\tr(X^*Y),
 \end{equation}
 for any  $X,Y \in \mathcal{B}_2(\h)$.
The Schatten $p$-ideals are equipped with a complete norm, making $\bp$ a Banach space for all $p \geq 1$ and, particularly for $p = 2$, a Hilbert space. These ideals are instrumental in spectral theory, quantum mechanics, and various approximation theories. For the sake of convenience, we denote by $\bp^+ = \bp \cap \mathcal{B}(\mathcal{H})^+$ the set of all positive operators in the ideal $\bp$.

The Schatten $p$-ideals for $1 < p < \infty$ exhibit remarkable geometric properties, being both uniformly convex and smooth. Uniform convexity ensures that sequences converging to the unit sphere also converge in norm, providing stability in functional analysis applications. Smoothness allows for unique supporting hyperplanes at each point on the unit sphere, which aids in establishing well-defined duality mappings. Additionally, the dual spaces of these ideals correspond to the Schatten $q$-ideals, where $q$ satisfies $\frac{1}{p} + \frac{1}{q} = 1$, a relationship that enhances their applicability in operator theory and the study of compact operators.

For an in-depth study of $p$-Schatten class operators, and more generally, ideals within the algebra $\mathcal{B}(\mathcal{H})$, we recommend that the reader consult \cite{G-K}.

Other geometric notions that can be explored in a normed space include parallelism and various types of orthogonality. In recent years, the authors of this manuscript have investigated both concepts and their implications within the context of $p$-Schatten ideals.  Based on the works \cite{Birk}, \cite{James}, and \cite{magajna}, we recall the following definitions for any $X, Y \in \bp$ with $p \geq 1$:  
\begin{enumerate}
	\item Birkhoff--James orthogonality: 
	$$X\obj Y \;\textrm{ if and only if}\;\| X + \gamma Y\|_p \geq \|X\|_p \text{~for all~} \gamma\in \mathbb{C}.$$
	
	\item $p$ Norm-paralellism:
$$X{\parallel}^p Y \;\textrm{ if and only if}\;
	\exists \lambda\in\mathbb{T} \:\textrm{such that}\:
	{\|X + \lambda Y\|}_p = {\|X\|}_p + {\|Y\|}_p,$$
	where $\mathbb{T}=\{\lambda \in \mathbb{C}: |\lambda|=1\}$.
\end{enumerate}

A complete characterization of $p$-norm parallelism was provided in \cite{BCMWZ}. In particular, Theorem 4.3 of \cite{BCMWZ} establishes the following condition:  
\begin{equation}\label{characterization p-parallelism}  
X {\parallel}^p Y \;\textrm{ if and only if }\;  
\|X\|_p |\tr(|X|^{p-1} U^* Y)| = \|Y\|_p \tr(|X|^p),
\end{equation}  
where $X=U|X|, Y\in \bp$.

This result offers a precise criterion for determining when two operators in $\bp$ satisfy $p$-norm parallelism, further enriching the geometric study of Schatten ideals.

Finally, we conclude this section of preliminary notions with an interesting property of $p$-norm parallelism that we will employ throughout the manuscript.

\begin{lemma} Let $X, Y\in \bp$ with $p\geq 1$. Then, 
		$X\parallel^pY$ if and only if $ rX\parallel^p sY$ for all $r, s\geq 0$.
\end{lemma}

We will omit the proof of the previous lemma, as it is an immediate consequence of \cite[Lemma 2.1]{AAB}.

\begin{remark}
	Moslehian and Zamani proved in \cite[Proposition 3.2]{ZM-2015},  that given $X, Y \in \bh$, then $X \parallel Y$ if and only if $r X \parallel s Y$ for all $r, s \in \mathbb{R}-\{0\}$. More precisely,  they showed that 
	$$
	\|X +\lambda Y\|=\|X\|+\|Y\|\:\: \textrm{if and only if } \:\: \|rX +\lambda sY\|=\|rX\|+\|sY\|,
	$$
	for some $\lambda \in \mathbb{T}$ and any $r, s \in \mathbb{R}-\{0\}$.

	  However, as illustrated by the following example, the previous statement is erroneous: let $P$ be a non-zero orthogonal projection. Then,	$
	\|I + P\| = 1 + \|P\| = 2.
	$ However, considering $r = 1$ and $s= -1$, we find that $
	\|I - P\| = 1 \neq 2.
	$
\end{remark}

In \cite{BCMWZ}, we explored the concept of Birkhoff-James orthogonality in the $\bp$ operator ideals. To do so, it was essential to derive an explicit expression for the semi-inner product. More specifically, we focused on providing a detailed formulation of this product to facilitate the study of orthogonality.

\begin{definition}\label{def_sip_bp}
In the ideal of operators $\bp$, for $1 \leq p < \infty$, we define the semi-inner product as follows:
\begin{equation}\label{semi inner prod def}
[Y, X]_p = \|X\|_p^{2-p} \operatorname{tr} \left( |X|^{p-1} U^* Y \right), 
\end{equation}
with  $X, Y \in \bp$ and $X = U|X|$ denotes a polar decomposition of $X$.
\end{definition}

It is worth noting that when $p=2$, the definition given in \eqref{semi inner prod def} coincides with that in \eqref{innerprod in b2}.

As we mentioned previously, Giles in \cite{giles} established that a semi-inner product, sharing several advantageous characteristics of the Hilbert space inner product, can be constructed within a wide range of Banach spaces. In particular, Giles focused on the class of smooth and uniformly convex Banach spaces, proving that these spaces admit a unique semi-inner product with the following properties. More precisely, he proved that a normed linear space $(\mathcal{X},\|\cdot\|)$ is smooth if and only if there exists a \emph{unique} semi-inner product (s.i.p.) that generates the norm $\|\cdot\|$; that is, the corresponding s.i.p.\ satisfies the first requirement in its definition (see \cite[Proposition~4]{Dra2004}). Moreover, this unique semi-inner product $[\,\cdot,\cdot\,]$ is automatically \emph{continuous}. Recall that an s.i.p.\ defined on a normed space $(\mathcal{X},\|\cdot\|)$ is said to be continuous whenever, for every $x,y\in\mathcal{X}$, one has
\begin{equation}\label{eq:continuity}
	\lim_{t\to 0} \operatorname{Re}[\,x,\, y+t x\,]
	\;=\; \operatorname{Re}[\,x,\,y\,].
\end{equation}
We refer the reader to see \cite{Dra2004}  for further details related to s.i.p.

By combining the previous results, we can assert that the semi-inner product $[\,\cdot,\cdot\,]_p$, introduced in Definition~\ref{semi inner prod def}, is the unique admissible form that generates the norm $\|\cdot\|_p$ for $1<p<\infty$, since the corresponding $p$-Schatten ideals are smooth Banach spaces. Furthermore, this unique semi-inner product is automatically continuous in the sense given in \eqref{eq:continuity}. 
 This implies that the definition provided here is independent of the particular polar decomposition employed. However, for the ideal of trace-class operators, this property does not hold, as will be demonstrated in the following example.

\begin{example}
	Let $A = \begin{pmatrix} 1 & 1 & 1 \\ 1 & 1 & 1 \\ 1 & 1 & 1 \end{pmatrix}$, which is a rank-one projector, as it can be written as $A = (1, 1, 1) \otimes (1, 1, 1)$. Clearly, $A \geq 0$. Any permutation of the identity matrix can be used to express the polar decomposition of $A$. For example, we have:
	
	\[
	A = I A = \underbrace{\begin{pmatrix} 0 & 1 & 0 \\ 0 & 0 & 1 \\ 1 & 0 & 0 \end{pmatrix}}_{U_1} \begin{pmatrix} 1 & 1 & 1 \\ 1 & 1 & 1 \\ 1 & 1 & 1 \end{pmatrix} = \underbrace{\begin{pmatrix} 0 & 0 & 1 \\ 1 & 0 & 0 \\ 0 & 1 & 0 \end{pmatrix}}_{U_2} \begin{pmatrix} 1 & 1 & 1 \\ 1 & 1 & 1 \\ 1 & 1 & 1 \end{pmatrix}.
	\]
	
	Note that the semi-inner product for $p = 1$ depends on the choice of partial isometry in the polar decomposition. For instance, if $D$ is any non-zero diagonal matrix, we observe that:
	
	\[
	\text{tr}(U_1^* D) = \text{tr}(U_2^* D) = 0, \quad \text{but} \quad \text{tr}(I D) \neq 0.
	\]
	
	This illustrates the dependence of the semi-inner product on the chosen partial isometry in the case where $p = 1$.
	
\end{example}

Based on the above discussion, we will dedicate the next sections to studying the properties of the semi-inner product (sip) for $1 < p < \infty$.

\section{$\bph$ as semi-inner product space with $1<p<\infty$}\label{s3}

If $\mathcal{X}$ is an inner product space, one of the first fundamental results that follows from its structure is the Cauchy–Schwarz inequality:
\begin{equation}\label{ineq: CS}
|\langle x, y \rangle |^2 \leq \langle x, x \rangle \langle y, y \rangle
\end{equation}
for all $x, y \in \mathcal{X}$.

Over the years, various refinements and generalizations of inequality \eqref{ineq: CS} have been developed by several authors. In what follows, we present a refinement that will be particularly useful for the results discussed in this section.

\begin{lemma}\label{C-S refinement in H}
	Let $\mathcal{X}$ be a  inner product space and $v,w,z\in \mathcal{X}$. If $z$ be such that $\|z\|=1$ and $\left\langle z,w \right\rangle=0 $, then
	$$|\left\langle v,w\right\rangle |^2+|\left\langle v,z\right\rangle |^2\|w\|^2\leq \|v\|^2\|w\|^2.$$
\end{lemma}
\begin{proof}
	In \cite{DraSan}, the authors proved the inequality  
	
	\begin{equation} 
	\left( \|v\|^2 \|z\|^2 - |\langle v, z \rangle|^2 \right)
	\left( \|w\|^2 \|z\|^2 - |\langle w, z \rangle|^2 \right)
	\geq
	\left( |\langle v, w \rangle| \|z\|^2 - \langle v, z \rangle \langle z, w \rangle \right)^2
	\end{equation}  
	
	for all $ v, w, z \in \mathcal{X} $.  
	
	In particular, for $ z \in \mathcal{X} $ with $ \|z\| = 1 $ and $ \langle z, w \rangle = 0 $,  we obtain 
	
	\begin{equation*}
	|\langle v, w \rangle|^2 + |\langle v, z \rangle|^2 \|w\|^2 \leq \|v\|^2 \|w\|^2.
	\end{equation*}

\end{proof}

A technical lemma is presented here, which will be essential throughout the manuscript and crucial for deriving a Cauchy-Schwarz-type inequality for the semi-inner product $[\cdot, \cdot]_p$.

\begin{lemma}\label{lem1}
	Let $X,Y\in \mathcal{B}_p(\mathcal{H})$ with $p>1$, and let $X=U|X|$ and $Y=V|Y|$ be their polar decompositions. Then:
	\begin{enumerate}
		\item $[|Y|,|X|]_p=\|X\|_p^{2-p}\,\tr(|X|^{p-1}|Y|)$ and 
		$[|Y^*|,|X^*|]_p=\|X\|_p^{2-p}\,\tr(|X^*|^{p-1}|Y^*|)$.
		
		\item $|X|^{\frac{p-1}{2}}|Y|^{1/2}\in \mathcal{B}_2(\h)$ and 
		\[
		\left\||X|^{\frac{p-1}{2}}|Y|^{1/2}\right\|_2^2
		= \frac{[|Y|,|X|]_p}{\|X\|_p^{2-p}}.
		\]
		
		\item $|X|^{\frac{p-1}{2}}U^*V|Y|^{1/2}\in \mathcal{B}_2(\mathcal{H})$ and 
		\[
		\left\||X|^{\frac{p-1}{2}}U^*V|Y|^{1/2}\right\|_2^2
		= \frac{[|Y^*|,|X^*|]_p}{\|X^*\|_p^{2-p}}.
		\]
		
		\item $[|Y|,|X|]_p\geq 0$ and $[|Y^*|,|X^*|]_p\geq 0$.
		
		\item 
		\[
		[Y,X]_p
		= \|X\|_p^{2-p}
		\left\langle |X|^{\frac{p-1}{2}}U^*V|Y|^{1/2},\;
		|X|^{\frac{p-1}{2}}|Y|^{1/2}\right\rangle_2.
		\]
	\end{enumerate}
\end{lemma}

\begin{proof}
	Item (1) follows directly from the definition of $[\cdot,\cdot]_p$ when applied to positive operators.
	
	To prove (2), note that
	\[
	\left\||X|^{\frac{p-1}{2}}|Y|^{1/2}\right\|_2^2
	= \tr\!\left(|Y|^{1/2}|X|^{p-1}|Y|^{1/2}\right)
	= \tr(|X|^{p-1}|Y|)
	= \frac{[|Y|,|X|]_p}{\|X\|_p^{2-p}} < \infty,
	\]
	so $|X|^{\frac{p-1}{2}}|Y|^{1/2}\in \mathcal{B}_2(\mathcal{H})$.
	
	For (3), using the polar decompositions and the properties in \eqref{prop_polar_decomp}, we obtain
	\begin{align*}
		\left\||X|^{\frac{p-1}{2}}U^*V|Y|^{1/2}\right\|_2^2
		&= \tr\!\left(|Y|^{1/2}V^*U|X|^{p-1}U^*V|Y|^{1/2}\right) \\
		&= \tr(|X^*|^{p-1}|Y^*|)
		= \frac{[|Y^*|,|X^*|]_p}{\|X^*\|_p^{2-p}} < \infty.
	\end{align*}
	Thus $|X|^{\frac{p-1}{2}}U^*V|Y|^{1/2}\in \mathcal{B}_2(\mathcal{H})$.
	
	Item (4) is now immediate, since both quantities in (2) and (3) are squared Hilbert--Schmidt norms and hence nonnegative.
	
	Finally, for (5) note that
	\begin{align*}
		[Y,X]_p
		&= \|X\|_p^{2-p}\,\tr(|X|^{p-1}U^*Y) \\
		&= \|X\|_p^{2-p}\,\tr\!\left(|Y|^{1/2}|X|^{\frac{p-1}{2}}|X|^{\frac{p-1}{2}}U^*V|Y|^{1/2}\right) \\
		&= \|X\|_p^{2-p}
		\left\langle 
		|X|^{\frac{p-1}{2}}U^*V|Y|^{1/2},\;
		|X|^{\frac{p-1}{2}}|Y|^{1/2}
		\right\rangle_2.
	\end{align*}
\end{proof}

%
%

As a consequence of the classical Cauchy–Schwarz inequality in the Hilbert space $\mathcal{B}_2(\mathcal{H})$, we obtain the following inequality. This result serves as an analogue of the Cauchy–Schwarz inequality in the context of semi-inner product spaces $\mathcal{B}_p(\mathcal{H})$, where we make use of the notion of adjoint operators naturally associated with these spaces.

\begin{theorem} \label{theo CS en bp}
	Let $X,Y\in \bp$, with $p>1$, such that $X=U|X|$ and $Y=V|Y|$. Then
\begin{equation} \label{CS en bp}
\left| [Y,X]_p \right|^2 \leq [ |Y^*|,|X^*|]_p [|Y|,|X|]_p.
\end{equation}
\end{theorem}

\begin{proof}
We observe that, by utilizing items (1) and (2) from Lemma \ref{lem1},
\begin{eqnarray*}
	\left|[Y,X]_p \right|^2&=&\|X\|_p^{4-2p}\left|\tr\left( |X|^{p-1}U^*V|Y|\right) \right|^2 \\
	&=&\|X\|_p^{4-2p}\left|\tr\left(\left( |Y|^{1/2} |X|^{\frac{p-1}{2}}\right) \left( |X|^{\frac{p-1}{2}}U^*V|Y|^{1/2}\right) \right) \right|^2 \\
	&=&\|X\|_p^{4-2p}\left|\left\langle |X|^{\frac{p-1}{2}}U^*V|Y|^{1/2}, |X|^{\frac{p-1}{2}}|Y|^{1/2}\right\rangle_2  \right|^2.
\end{eqnarray*}
	
	Since the ideal $\mathcal{B}_2(\h)$ is a Hilbert space, it specifically satisfies the Cauchy-Schwarz inequality. Consequently, we have that
	
	\begin{eqnarray*}
	\left|[Y,X]_p \right|^2&=&\|X\|_p^{4-2p}\left|\left\langle |X|^{\frac{p-1}{2}}U^*V|Y|^{1/2}, |X|^{\frac{p-1}{2}}|Y|^{1/2}\right\rangle_2  \right|^2\\
	&\leq& \|X\|_p^{4-2p}\left\| |X|^{\frac{p-1}{2}}U^*V|Y|^{1/2}\right\| _2^2\left\|  |X|^{\frac{p-1}{2}}|Y|^{1/2}\right\| _2^2\\
	&=&\|X\|_p^{4-2p}\tr\left(U|X|^{p-1}U^*V|Y|V^*\right)\tr\left(|X|^{p-1}|Y| \right).
\end{eqnarray*}

Taking into account the properties of the polar decomposition presented in \eqref{prop_polar_decomp}, it follows that

\begin{eqnarray*}
	\left|[Y,X]_p \right|^2
	&\leq&\|X\|_p^{4-2p}\tr\left(|X^*|^{p-1}|Y^*|\right)\tr\left(|X|^{p-1}|Y| \right)  \\
	&=&\|X\|_p^{4-2p}\dfrac{[|Y^*|,|X^*|]_p}{\|X^*\|_p^{2-p}}\dfrac{[|Y|,|X|]_p}{\|X\|_p^{2-p}}\\
	&=& [ |Y^*|,|X^*|]_p [|Y|,|X|]_p,
\end{eqnarray*}

and this concludes the proof.
\end{proof}

%

In what follows, we present a refinement of inequality \eqref{CS en bp}, derived from Lemma~\ref{C-S refinement in H}.

\begin{theorem} \label{teo refinement CS}
	Let $X,Y\in \bp$, with $p>1$, such that $X=U|X|$ and $Y=V|Y|$. Suppose that there exists $Z\in\mathcal{B}_2(\h)$ such that
	\begin{enumerate}
		\item $\|Z\|_2=1$.
		\item $\left\langle Z,|X|^{\frac{p-1}{2}}|Y|^{1/2} \right\rangle_2=0 $.
	\end{enumerate}
	Then,
	\begin{equation} \label{CS refin en bp}
	\left| [Y,X]_p\right| ^2+\delta_p(Z) \leq [ |Y^*|,|X^*|]_p [|Y|,|X|]_p, 
	\end{equation}
	where $$\delta_p(Z)=\left|\left\langle |X|^{\frac{p-1}{2}}U^*V|Y|^{1/2},Z \right\rangle_2  \right|^2\left\| |X|^{\frac{p-1}{2}}|Y|^{1/2}\right\| _2^2 \|X\|_p^{4-2p} \geq 0. $$
\end{theorem}

\begin{proof}
	
	First, note that if either $X$ or $Y$ equals $0$, inequality \eqref{CS refin en bp} is trivially satisfied. Suppose now that both $X$ and $Y$ are non-zero. In this case, we observe that $\delta_p$ is well-defined by item (2) of Lemma \ref{lem1}.
	Then, we have 
	\begin{eqnarray*}
		\left|[Y,X]_p \right|^2+\delta_p(Z)	&=&\|X\|_p^{4-2p}\left[ \left|\left\langle |X|^{\frac{p-1}{2}}U^*V|Y|^{1/2}, |X|^{\frac{p-1}{2}}|Y|^{1/2}\right\rangle_2  \right|^2\right. \\
		& &\left. +\left|\left\langle |X|^{\frac{p-1}{2}}U^*V|Y|^{1/2},Z \right\rangle_2  \right|^2\left\| |X|^{\frac{p-1}{2}}|Y|^{1/2}\right\| _2^2\right]. 
			\end{eqnarray*}
	Considering  $v = |X|^{\frac{p-1}{2}} U^* V |Y|^{\frac{1}{2}}$, $w = |X|^{\frac{p-1}{2}} |Y|^{\frac{1}{2}}$, and  $z = Z$ in Lemma \ref{C-S refinement in H}, we obtain from hypothesis (2) and Lemma \ref{lem1} that
		
		\begin{eqnarray*}
	\left|[Y,X]_p \right|^2+\delta_p(Z)	&\leq&\|X\|_p^{4-2p}\left\| |X|^{\frac{p-1}{2}}U^*V|Y|^{1/2}\right\| _2^2\left\| |X|^{\frac{p-1}{2}}|Y|^{1/2}\right\| _2^2\\
		&=&\|X\|_p^{4-2p}\tr\left(|X^*|^{p-1}|Y^*|\right)\tr\left(|X|^{p-1}|Y| \right)  \\
		&=&\|X\|_p^{4-2p}\dfrac{[|Y^*|,|X^*|]_p}{\|X^*\|_p^{2-p}}\dfrac{[|Y|,|X|]_p}{\|X\|_p^{2-p}}\\
		&=& [ |Y^*|,|X^*|]_p [|Y|,|X|]_p.
	\end{eqnarray*}
\end{proof}

The following is a simple and well-known result that will be useful for establishing conditions for the existence of an operator $Z \in \mathcal{B}_2(\mathcal{H})$ satisfying Theorem~\ref{teo refinement CS}.

\begin{lemma} \label{lem trAB}
	For every $A,B\in \mathcal{B}(\h)^+$, the following conditions are equivalent:
\begin{enumerate}
	\item 	$\tr(AB)=0$.
	\item  $\|B^{1/2}A^{1/2}\|_2=0$.
	\item  $B^{1/2}A^{1/2}=0$.
	\item $BA=0$ ($AB=0$).
\end{enumerate}	
\end{lemma}
\begin{proof}
All the implicances can be deduced by observing that
$$0=\tr(AB)=\tr(A^{1/2}B^{1/2}B^{1/2}A^{1/2})=\tr((B^{1/2}A^{1/2})^*B^{1/2}A^{1/2})=\|B^{1/2}A^{1/2}\|_2^2.$$
\end{proof}

\begin{remark}
	Let $X,Y\in \bp$, with $p>1$, such that $X=U|X|$ and $Y=V|Y|$. 	If $|X|^{p-1}|Y|=0$, then $$\delta_p(Z)=[Y,X]_p=0$$
	for every $Z\in \mathcal{B}_2(\h)$.
	
Indeed, by Lemma \ref{lem trAB},  $|X|^{p-1}|Y|=0$ is equivalent to $|X|^{\frac{p-1}{2}}|Y|^{1/2}=0 $, since $|X|^{p-1}$ and $|Y|$ are positive operators.  Then, $\delta_p(Z)=0$ for every $Z\in \mathcal{B}_2(\h)$ and 
$[Y,X]_p=\|X\|_p^{2-p}\tr(|Y||X|^{p-1}U^*V)=0$.
\end{remark}

For the previous remark, we assume that $|X|^{\frac{p-1}{2}}|Y|^{1/2}\neq 0$.

\begin{remark}
	
Since $B_2(\mathcal{H})$ is a Hilbert space, as a consequence of the notion of orthogonal complement in such a space, we can prove that given $X, Y \in B_p(\mathcal{H})$, there exists $Z \in B_2(\mathcal{H})$ that satisfies the hypotheses of Theorem \ref{teo refinement CS}. Specifically, recall that if $\mathcal{W} \subset B_2(\mathcal{H})$, then

\[
\mathcal{W}^\perp = \{ Z \in B_2(\mathcal{H}) : \langle Z, W \rangle_2 = 0, \forall\: W \in \mathcal{W} \}.
\]

Now, considering $\mathcal{W} = \left\{ |X|^{\frac{p-1}{2}} |Y|^{\frac{1}{2}} \right\}$, we know that $\mathcal{W}^\perp$ is non-empty and non-trivial. Therefore, by taking $Z_0 \in \mathcal{W}^\perp$ with $Z_0 \neq 0$, it follows that $Z = \frac{Z_0}{\|Z_0\|}$ satisfies the desired conditions

Observe that \eqref{CS refin en bp} refines \eqref{CS en bp} if there exists $Z \in \mathcal{B}_2(\h)$, $Z \neq 0$, orthogonal to $|X|^{\frac{p-1}{2}} |Y|^{\frac{1}{2}}$ and satisfying $\delta_p(Z) > 0$. This means that such a $Z$ must be orthogonal to $|X|^{\frac{p-1}{2}} |Y|^{\frac{1}{2}}$ but not to $|X|^{\frac{p-1}{2}} U^* V |Y|^{\frac{1}{2}}$ with respect to the Hilbert-Schmidt inner product. In conclusion, we obtain a genuine refinement of \eqref{CS refin en bp} if and only if
\begin{equation}\label{intersection}
\left\{ |X|^{\frac{p-1}{2}} |Y|^{\frac{1}{2}} \right\}^{\perp} \setminus \left\{ |X|^{\frac{p-1}{2}} U^* V |Y|^{\frac{1}{2}} \right\}^{\perp} \neq \emptyset.
\end{equation}

\end{remark}

\begin{example}
	
	\begin{enumerate}
%
%
%
%
%
%
%

	\item
	Let us consider matrices $X, Y \in \mathbb{R}^{2 \times 2}$ defined as
	\[
	X = \begin{pmatrix} 2 & 0 \\ 0 & 1 \end{pmatrix} \quad \textrm{and}\quad  Y = \begin{pmatrix} 0 & 1 \\ 1 & 0 \end{pmatrix}.
	\]
	
	Since $X$ is positive, its polar decomposition is
	\(
	X = U |X|,\) with  $U= I$ and 
	$|X| = X$.
	
	For $Y$, note that $Y$ is unitary, so its polar decomposition is given by
	\(
	Y = V |Y|\),  with  $V = Y$  and  $|Y| = I.$
	Thus, we have
	\[
	U^* V = V = Y = \begin{pmatrix} 0 & 1 \\ 1 & 0 \end{pmatrix} \neq I.
	\]
	
	For any given $p>1$, define
	\[
	A = |X|^{\frac{p-1}{2}} |Y|^{\frac{1}{2}} \quad \text{and} \quad B = |X|^{\frac{p-1}{2}} U^* V|Y|^{\frac{1}{2}}.
	\]
	Since $|Y| = I$, we obtain
	\[
	A = |X|^{\frac{p-1}{2}} = \begin{pmatrix} 2^{\frac{p-1}{2}} & 0 \\ 0 & 1 \end{pmatrix},
	\]
	and
	\[
	B = |X|^{\frac{p-1}{2}} U^* V|Y|^{\frac{1}{2}}. = \begin{pmatrix} 2^{\frac{p-1}{2}} & 0 \\ 0 & 1\end{pmatrix} \begin{pmatrix} 0 & 1 \\ 1 & 0 \end{pmatrix} = \begin{pmatrix} 0 & 2^{\frac{p-1}{2}} \\ 1 & 0 \end{pmatrix}.
	\]

	To find the orthogonal complement of $A$ with respect to the Hilbert-Schmidt inner product, we consider:
	\[
	\{A\}^\perp = \left\{ Z \in \mathbb{R}^{2 \times 2} \mid \tr(Z^* A) = 0 \right\}.
	\]

	For any $Z = \begin{pmatrix} z_{11} & z_{12} \\ z_{21} & z_{22} \end{pmatrix} $, we compute:
	\[
	Z^* A = \begin{pmatrix} z_{11} & z_{21} \\ z_{12} & z_{22} \end{pmatrix}\begin{pmatrix} 2^{\frac{p-1}{2}} & 0 \\ 0 & 1 \end{pmatrix}
	= \begin{pmatrix} z_{11} 2^{\frac{p-1}{2}} & z_{21}  \\ z_{12} 2^{\frac{p-1}{2}} & z_{22} \end{pmatrix}.
	\]
	Taking the trace, we obtain 
	\[
	\tr(Z^* A) = z_{11} 2^{\frac{p-1}{2}} + z_{22}  = 0.
	\]
	Thus, the orthogonal complement of $A$ is:
	\[
	\{A\}^\perp = \left\{ Z = \begin{pmatrix} z_{11} & z_{12} \\ z_{21} & z_{22} \end{pmatrix} :  z_{11} 2^{\frac{p-1}{2}} + z_{22} = 0\right\}.
	\]
	
Similarly, for $\{B\}^\perp$, we compute the following. Thus, the orthogonal complement of $B$ is given by:
	\[
	\{B\}^\perp = \left\{ Z = \begin{pmatrix} z_{11} & z_{12} \\ z_{21} & z_{22} \end{pmatrix} :  z_{21} +  z_{12}  2^{\frac{p-1}{2}} = 0 \right\}.
	\]
	
	Now, we verify that the difference $\{A\}^\perp \setminus \{B\}^\perp$ is nonempty. Consider the matrix:
	\[
	Z_0 = \begin{pmatrix} -1 & 1 \\ 1 & 2^{\frac{p-1}{2}} \end{pmatrix}.
	\]
	This satisfies
	\(
	(-1) 2^{\frac{p-1}{2}} + 2^{\frac{p-1}{2}}= 0,
	\)
	so $Z_0 \in \{A\}^\perp$. However, 
	\(
	  1+    2^{\frac{p-1}{2}} \neq 0,
	\)
	 meaning $Z_0 \notin \{B\}^\perp$. Hence,
	\[
	\left\{ |X|^{\frac{p-1}{2}}|Y|^{\frac{1}{2}} \right\}^{\perp}\setminus \left\{ |X|^{\frac{p-1}{2}}U^* V|Y|^{\frac{1}{2}} \right\}^{\perp} \neq \emptyset,
	\]
	
and
\begin{eqnarray*}
	\delta_p(Z_0)&=&\left|\left\langle B,Z_0 \right\rangle_2  \right|^2\left\| A\right\| _2^2 \|X\|_p^{4-2p}\\
	&=&(2^{\frac{p-1}{2}}+1)^2\tr\begin{pmatrix}
		2^{p-1}&0\\
		0&1
	\end{pmatrix}(2^p+1)^{\frac{4-2p}{p}}\\
&=&(2^{\frac{p-1}{2}}+1)^2(2^{p-1}+1)(2^p+1)^{\frac{4-2p}{p}}>0
\end{eqnarray*}
	
\item

Given a fixed value of $p > 1$, we seek to construct two operators $X, Y \in \bp $   satisfy the geometric condition:
\[
\left\{ |X|^{\frac{p-1}{2}}|Y|^{\frac{1}{2}} \right\}^{\perp}\setminus \left\{ |X|^{\frac{p-1}{2}}U^* V|Y|^{\frac{1}{2}} \right\}^{\perp} \neq \emptyset.
\]

Let $\{e_n\}_{n\geq1}$ be an orthonormal basis  of a Hilbert space $\h$. We consider  $X$to be diagonal in the following way:
\[
X e_n = \frac{(-1)^n}{n^\alpha} e_n=\lambda_n e_n,
\]
with $\alpha > \frac{1}{p}.$ To show that the operator $X$ belongs to the $p$ Schatten -class $\mathcal{B}_p(\h)$, we need to compute its $p$-norm, which is given by
\[
\|X\|_p = \left( \sum_{n=1}^{\infty} |\lambda_n|^p \right)^{1/p}=  \left( \sum_{n=1}^{\infty} \left| \frac{(-1)^n}{n^\alpha} \right|^p \right)^{1/p} = \left( \sum_{n=1}^{\infty} \frac{1}{n^{\alpha p}} \right)^{1/p}<\infty, 
\] 
because  $\alpha p >1.$

Since $X$ is diagonal, its polar decomposition is immediate:
\[
X = U |X|, \quad \text{with } |X| e_n = |\lambda_n| e_n \text{ and } U e_n = \operatorname{sgn}(\lambda_n) e_n.
\]
We note that  $U$ is a unitary operator acting as a reflection in the canonical basis.

Now we define $Y$ in the following way
\[
Y e_n = \mu_n e_{n+1},
\]
where the values $\mu_n$ are chosen as
\(
\mu_n = \frac{1}{(n+1)^\alpha}\) with  $\alpha > \frac{1}{p}.$

Again, we verify that $Y \in \bp$:
\[
\|Y\|_p^p=\sum_{n=1}^{\infty} |\mu_n|^p = \sum_{n=1}^{\infty} \frac{1}{(n+1)^{\alpha p}} < \infty.
\]

The polar decomposition of $Y$ is given by \(
Y=V|Y|, 
\) where  $Ve_n=e_{n+1}$ and 
$ |Y| e_n = \mu_n e_n.$

We aim to verify that the set
\[\left\{ |X|^{\frac{p-1}{2}}|Y|^{\frac{1}{2}} \right\}^{\perp} \setminus \left\{ |X|^{\frac{p-1}{2}}U^* V|Y|^{\frac{1}{2}} \right\}^{\perp} \subseteq \mathcal{B}_2(\mathcal{H}),
\]
is non-empty. To demonstrate this, we provide an explicit element in the set difference.

Let $Z$ be a shift operator defined by:
	
	\[
	Z e_n =  \frac{(-1)^{n+1}}{n^\gamma}e_{n+1}=\nu_n e_{n+1},
	\]
	where $\gamma > \frac{1}{2}.$

	Thus, the Hilbert-Schmidt norm of $Z$ is
	\[
	\|Z\|_{2}^2 = \sum_{n=1}^{\infty} \|Z e_n\|^2=\sum_{n=1}^{\infty} \left| \frac{(-1)^{n+1}}{n^\gamma} \right|^2=\sum_{n=1}^{\infty} \frac{1}{n^{2\gamma}}<\infty, 
	\]
	
	since $
	\gamma > \frac{1}{2}.$

	Now, we need to verify that $Z$ belongs to the orthogonal complement $\left\{ |X|^{\frac{p-1}{2}} |Y|^{\frac{1}{2}} \right\}^{\perp}$. For this, we have 
	
	\begin{eqnarray}
	\langle Z, |X|^{\frac{p-1}{2}} |Y|^{\frac{1}{2}}\rangle_2&=&\sum_{n=1} ^{\infty}\langle Ze_n, |X|^{\frac{p-1}{2}} |Y|^{\frac{1}{2}} e_n\rangle\nonumber \\
	&=&\sum_{n=1} ^{\infty}\left\langle  \frac{(-1)^n}{n^\gamma}e_{n+1}, \frac{1}{n^{\alpha \cdot \frac{p-1}{2}} (n+1)^{\frac{\alpha}{2}}} e_n\right\rangle=0. \nonumber \
	\end{eqnarray}

	Next, we verify that $Z$ does not belong to the orthogonal complement $\left\{ |X|^{\frac{p-1}{2}} U^* V |Y|^{\frac{1}{2}} \right\}^{\perp}$. 	We know that:

	\begin{eqnarray*}
\langle Z, |X|^{\frac{p-1}{2}} U^* V |Y|^{\frac{1}{2}}\rangle_2&=&
	\sum_{n=1}^{\infty}\langle Z e_n, |X|^{\frac{p-1}{2}} U^* V |Y|^{\frac{1}{2}} e_n \rangle  \nonumber \\ &=&
		\sum_{n=1}^{\infty} \frac{(-1)^{2n+1}}{n^\gamma} \frac{1}{n^{\alpha \frac{p-1}{2}}} \frac{1}{(n+1)^{\alpha \frac{1}{2}}} <0.
	\end{eqnarray*}	
	This demonstrates that $Z$ belongs to the difference of the orthogonal complements sought.

\end{enumerate}
	
\end{example}

\section{Angles  in $\bp$}\label{s4}

\subsection{A First Approach to the Notion of Angle}

A fundamental aspect of inner product spaces is the ability to define an angle between vectors, which provides deep geometric insight and facilitates the study of orthogonality, projections, and operator theory. This notion is inherently tied to the inner product, which allows for a rigorous and consistent measurement of the relative orientation between elements. The concept of angle has proven to be a powerful tool in various mathematical frameworks, particularly in functional analysis and spectral theory.

In this section we develop a parallel theory of angles in semi‐inner product spaces.  Because a semi‐inner product need not satisfy conjugate symmetry, the classical construction of angles must be carefully revised.  Our goal is to understand how this more flexible structure alters fundamental geometric properties and to extend key results from inner product spaces to this broader context.  We begin by introducing the following definitions, in direct analogy with the inner product setting.

\begin{definition} \label{semiangles}
Let $p> 1$ and $X,Y\in \bp$. We define the following quantities
\begin{equation}\label{def alpha}
\alpha_{Y,X}=\dfrac{[Y,X]_p}{\|X\|_p\|Y\|_p}
\end{equation}
and 
\begin{equation}\label{def beta}
	\beta_{Y,X}={\rm Re}(\alpha_{Y,X}).
\end{equation}

\end{definition}

We now collect the fundamental properties of the scalar $\alpha_{Y,X}$, which encodes the semi‐inner product relationship between $X$ and $Y$ in $\mathcal{B}_p(\mathcal{H})$. The following proposition lists its "homogeneity" under scalar multiplication, normalization, boundedness, and its characterization of Birkhoff–James orthogonality.

\begin{proposition} \label{properties of alpha}
	
Let $p> 1$ and $X,Y\in \bp$. Then,
\begin{enumerate}
	\item \label{prop escalar} $\alpha_{aY,bX}=\frac{\overline{b}a}{|a||b|}\  \alpha_{Y,X}$, for every $a,b\in \C-\{0\}$.
	\item $a\alpha_{Y,X}=|a|\alpha_{aY,X}$ for every $a\in \C-\{0\}$.
	\item $\alpha_{X,X}=1$, for all $X\neq 0$.
	\item $|\alpha_{Y,X}|\leq 1$.
	\item \label{prop ortogonal} $\alpha_{Y,X}=0$ if and only if $X\obj Y$.
\end{enumerate}

\end{proposition}

\begin{proof}
Items (1) to (4) are direct consequences from definition \eqref{def alpha}. 

(5) is valid by the equivalence $X\obj Y \Leftrightarrow [Y,X]_p=0$, proved in Theorem 3.8 in \cite{BCMWZ}.

\end{proof}

The next result characterizes the cases in which the modulus of $\alpha_{Y,X}$ attains its maximal possible value. This occurs precisely when $X$ and $Y$ are $p$-norm parallel.

\begin{theorem} \label{|alpha|=1}
	
Let $X,Y\in \bp-\{0\}$ and $p>1$. The following conditions are equivalent:
\begin{enumerate}
	\item\label{1} $|\alpha_{Y,X}|= 1$.
	\item \label{2}$X\parallel^p Y$.
	\item \label{3}$Y\parallel^pX$.
	\item \label{4}$|\alpha_{X,Y}|= 1$.
\end{enumerate}

\end{theorem}

\begin{proof}
The equivalence between the items \eqref{2} and \eqref{3} is immediate, as the notion of parallelism is symmetric. For this reason, it is sufficient to show that the first two items are equivalent. Now, we assume that item \eqref{2} holds. Then, by  \eqref{characterization p-parallelism} we have 
\begin{equation*}
\|X\|_p\left|\tr\left( |X|^{p-1}U^* Y\right)  \right|=\|Y\|_p\tr(|X|^p),
\end{equation*}
or equivalently, 
\begin{equation*}
\|X\|_p^{1-p}\left|\tr\left( |X|^{p-1}U^* Y\right)  \right|=\|Y\|_p.
\end{equation*}
Multiplying both terms of the equality by $\|X\|_p$, we obtain 
\begin{equation*}
\|X\|_p^{2-p}\left|\tr\left( |X|^{p-1}U^* Y\right)  \right|=\|Y\|_p\|X\|_p.
\end{equation*}
This means that, $|[Y,X]_p|=\|Y\|_p\|X\|_p$, and thus \eqref{1} is satisfied.  Now, if \eqref{1} holds true, then $|[Y,X]_p|=\|Y\|_p\|X\|_p $,  which implies that $X\parallel^pY$. 
By an analogous argument, we prove that \eqref{3} and \eqref{4} are equivalent, and thus the proof is complete.
\end{proof}

%
%
%
%

Let us note that if  $X, Y\in \bp$ with $p>1$ and $X, Y$ linearly dependent, by Theorem \ref{|alpha|=1}, it follows that $|\alpha_{Y,X}| = |\alpha_{X,Y}| = 1$. In the following corollary, we also show that if $Y = aX$ with $a \in \mathbb{R}$, then $\alpha$ is symmetric. More precisely,

\begin{corollary}
Let $X, Y\in \bp$ with $p> 1$  such that $Y=aX$ for some $a\in \R$, then $\alpha_{Y,X}=\alpha_{X,Y}$.
\end{corollary}

\begin{proof}
Observe that from Proposition \ref{properties of alpha} we obtain
$$\alpha_{Y,X}=\alpha_{aX,X}=\frac{a}{|a|}\alpha_{X,X}\ \text{and } \alpha_{X,Y}=\alpha_{X,aX}=\frac{\overline{a}}{|a|}\alpha_{X,X}.$$
Since $a\in \R$, we conclude that $\alpha_{Y,X}=\alpha_{X,Y}$.
\end{proof}

We will prove later that the converse of this corollary is not true. See Example \ref{ej no trivial} for clarification.

\bigskip

We now derive a Cauchy–Schwarz-type inequality for $\alpha$ and $\beta$, where the bounds are expressed in terms of the moduli of the associated operators.

\begin{theorem}
Let $X,Y\in \bp$, $p> 1$, then we have 
\begin{equation} \label{cota CS alpha}
|\alpha_{Y,X}|^2\leq \alpha_{|Y^*|,|X^*|}\alpha_{|Y|,|X|}
\end{equation}
and
\begin{equation} \label{cota CS beta}
\beta_{Y,X}^2\leq \beta_{|Y^*|,|X^*|}\beta_{|Y|,|X|}.
\end{equation}

\end{theorem}

\begin{proof}
Taking the inequality in \eqref{CS en bp} and dividing both sides by $\|X\|_p^2\|Y\|_p^2$, we obtain that
$$\dfrac{|[Y,X]_p|^2}{\|X\|_p^2\|Y\|_p^2}\leq \dfrac{[|Y^*|,|X^*|]_p}{\|X\|_p\|Y\|_p} \dfrac{[|Y|,|X|]_p}{\|X\|_p\|Y\|_p}. $$
Therefore, 
$$|\alpha_{Y,X}|^2\leq \alpha_{|Y^*|,|X^*|}\alpha_{|Y|,|X|}.$$
On the other hand, 
$$\beta_{Y,X}^2={\rm Re}(\alpha_{Y,X})^2\leq |\alpha_{Y,X}|^2\leq \alpha_{|Y^*|,|X^*|}\alpha_{|Y|,|X|}={\rm Re}\left( \alpha_{|Y^*|,|X^*|}\right) {\rm Re}\left( \alpha_{|Y|,|X|}\right).$$

\end{proof}

In general, $[Y,X]_p \neq [X,Y]_p$, and therefore $\alpha_{Y,X} \neq \alpha_{X,Y}$; that is, the relation is non-symmetric. However, when $X$ and $Y$ have disjoint supports—which implies $XY^* = Y^*X = 0$ (see \cite{BCMWZ})— then
\begin{equation}\label{alpha=0}
	\alpha_{Y,X}=0=\alpha_{X,Y}.
\end{equation}
In this particular case (of disjoint support), we have simultaneously that $X\obj Y$ and $Y\obj X$.

The following result, which stems directly from Proposition \ref{properties of alpha}, item \eqref{prop ortogonal}, establishes that the mutual Birkhoff-James $p$-orthogonality between two operators is both a necessary and sufficient condition for the identity \eqref{alpha=0} to hold.

\begin{corollary} \label{coro equiv}
	Let $p> 1$ and $X,Y\in \bp$. Then, the following conditions are equivalent:
	\begin{enumerate}
		\item $\alpha_{Y,X}=\alpha_{X,Y}=0$.
		\item $X\obj Y$ and $Y\obj X$.
	\end{enumerate}
\end{corollary}

In the context of positive operators, we can derive \eqref{alpha=0} by relaxing the hypothesis of the previous corollary.

\begin{proposition}\label{prop: positivos y ortogonalidad implican alpha igual a cero}
Let $X,Y\in\bp^+$, $p>1$ such that $X\obj Y$. Then,
$\alpha_{Y,X}=\alpha_{X,Y}=0$.
\end{proposition}

\begin{proof}
By Proposition 3.1 and Theorem 3.2 in \cite{BCMWZ}, if $X\obj Y$ then $Y\obj X$ and we conclude using Corollary \ref{coro equiv} that $\alpha_{Y,X}=\alpha_{X,Y}=0$.
\end{proof}

%
%
%

Next, in the finite‐dimensional setting and under the assumption that one of the operators is the identity, we give necessary and sufficient conditions for $
	\alpha_{I,X} = \alpha_{X,I}
	$,
	that is, for the scalar $\alpha$ to be symmetric.

\begin{proposition} \label{equiv zero trace}
Let $Z$ be in the algebra $M_n(\C)$ of all complex $n\times n$ matrices, with polar decomposition $Z=U|Z|$ and $p> 1$. Then the following statements are equivalent:
\begin{enumerate}
	\item\label{11} $[Z,I]_p=[I,Z]_p$.
	\item \label{22} $ \tr\left( \left( n^{ \frac 2p-1}|Z|-\|Z\|_p^{2-p}|Z|^{p-1}\right)U^*\right) =0 $.
	\item \label{33} $I\obj \left( n^{ \frac 2p-1}|Z|-\|Z\|_p^{2-p}|Z|^{p-1}\right)U^* $.
\end{enumerate}
\end{proposition}

\begin{proof}
	\eqref{11} $\Leftrightarrow$ \eqref{22} Let us observe that if the equality $[Z,I]_p=[I,Z]_p$ holds for some $n\in \mathbb{N}$, $p> 1$ and some matrix $Z\in M_n(\C)$ then 
$$ \|I\|_p\tr\left(Z \right)=\|Z\|_p^{2-p}\tr\left(|Z|^{p-1}U^* \right),$$
or equivalently,  by the additivity and homogeneity of the trace, we have 
$$\tr\left(n^{\frac2p-1}Z-\|Z\|_p^{2-p}|Z|^{p-1}U^* \right)=0.$$

Using the polar decomposition of $Z=U|Z|$, we obtain 

$$\tr\left(\left( n^{\frac2p-1}|Z|-\|Z\|_p^{2-p}|Z|^{p-1}\right) U^* \right)=0.$$

\eqref{22} $\Leftrightarrow$ \eqref{33} We recall that, since $\left( n^{\frac2p-1}|Z|-\|Z\|_p^{2-p}|Z|^{p-1}\right) U^*$ is a matrix with trace zero, it follows  by \cite{Kitta-zero trace} and Theorem 3.8 in \cite{BCMWZ} this is equivalent to 
$$I\perp^p\left( n^{\frac2p-1}|Z|-\|Z\|_p^{2-p}|Z|^{p-1}\right) U^*.$$

\end{proof}

If $p\geq 2$, we additionally obtain that 
\begin{equation}\label{equiv}
[Z,I]_p=[I,Z]_p \quad\textrm{if and only if}\quad  \tr\left( \left( n^{ \frac 2p-1}I-\|Z\|_p^{2-p}|Z|^{p-2}\right)Z\right) =0.
\end{equation}

The following example illustrates the existence of non‑trivial matrices that commute with the identity under the $[\cdot,\cdot]_p$ bracket.

\begin{example}\label{ej no trivial}

		
	We consider the case of an even integer $n$ and $p > 1$, let  
	$$Z = \text{Diag}(a, -a, a, -a, \ldots, a, -a) \in M_n(\C),$$ 
	where $a > 0$. We can see that \(\|I\|_p = n^{\frac{2}{p} - 1}\), \(\|Z\|_p^{2 - p} = n^{\frac{2}{p} - 1}a^{2 - p}\), and \(|Z| = aI\). Therefore:
		\[
		\mathrm{tr}\left( \left( n^{\frac{2}{p} - 1}I - n^{\frac{2}{p} - 1}a^{2 - p}a^{p - 2}I \right) Z \right) = 0.
		\]
		Then, by Proposition \ref{equiv zero trace} and equation \eqref{equiv}, this leads to $[I, Z]_p = [Z, I]_p$.
		
	 Lastly, for an odd integer $n$ and $p \geq 1$, we consider
	  $$Z = \text{Diag}(a, -a, a, -a, \ldots, a - a, 0) \in \mathbb{C}^{n \times n}$$ with \(a > 0\), we observe that $\|I\|_p = n^{\frac{2}{p} - 1}$ and $\|Z\|_p^{2 - p} = n^{\frac{2}{p} - 1}a^{2 - p}$. Thus:
		\begin{align*}
		\mathrm{tr}\left( \left( n^{\frac{2}{p} - 1}I - \|Z\|_p^{2 - p}|Z|^{p - 2} \right) Z \right) &= \sum_{i = 1}^{n - 1} \left( n^{\frac{2}{p} - 1} - n^{\frac{2}{p} - 1}a^{2 - p}a^{p - 2} \right) Z_{ii} \\
		& \quad + \left( n^{\frac{2}{p} - 1} - n^{\frac{2}{p} - 1}a^{2 - p} \cdot 0 \right) Z_{nn} \\
		&= 0.
		\end{align*}
		This equivalence confirms that $[I, Z]_p = [Z, I]_p$ for any $n\in \mathbb{N}$ and $p> 1.$

\end{example}

In the previous example, any matrix $Z$ satisfying
\[
[I, Z]_p = [Z, I]_p
\]
for $p > 1$ must have $\operatorname{tr}(Z)=0$. However, the following example shows that this trace condition alone does not suffice to ensure symmetry of the semi‑inner product.

Let

\[
Z = \begin{pmatrix} 
1 & 2 \\
-2 & -1
\end{pmatrix} 
= U|Z| 
= \begin{pmatrix} 
\frac{1}{\sqrt{5}} & \frac{2}{\sqrt{5}} \\
-\frac{2}{\sqrt{5}} & -\frac{1}{\sqrt{5}}
\end{pmatrix} 
\begin{pmatrix} 
\sqrt{5} & 0 \\
0 & \sqrt{5} 
\end{pmatrix}.
\]

Then,

\[
[Z, I]_p = \|I\|_p^{2-p} \operatorname{tr} \left(Z \right) = 0.
\]

Next, we calculate $[I, Z]_p$. Using the polar decomposition $Z = U |Z| $, we have:

\[
|Z|^{p-1} U^* = \begin{pmatrix} 
\left( \sqrt{5} \right)^{p-1} & 0 \\
0 & \left( \sqrt{5} \right)^{p-1}
\end{pmatrix}
\begin{pmatrix} 
\frac{1}{\sqrt{5}} & -\frac{2}{\sqrt{5}} \\
\frac{2}{\sqrt{5}} & \frac{1}{\sqrt{5}}
\end{pmatrix}
= \begin{pmatrix} 
\left( \sqrt{5} \right)^{p-2} & -2 \left( \sqrt{5} \right)^{p-2} \\
2 \left( \sqrt{5} \right)^{p-2} & \left( \sqrt{5} \right)^{p-2}
\end{pmatrix}.
\]

Therefore,
\[
[I, Z]_p = \|Z\|_p^{2-p} \operatorname{tr} \left( |Z|^{p-1} U^* \right) =  2 \left( \sqrt{5} \right)^{p-2}\|Z\|_p^{2-p}.
\]

Thus, we have demonstrated that $ Z$ is a matrix with trace zero such that $
[Z, I]_p \neq [I, Z]_p$ and this inequality holds for any $p > 1$.


\subsection{Different notions of angles in the space \(\bp\)}

	To conclude this manuscript, we explore various notions of angles in the space $\bp$, defined through the properties of the semi-inner product. These notions are designed to recover and reflect key geometric features of the ideal. As a starting point, we briefly recall several definitions of angles that have been introduced over the past decades in the context of real normed spaces. This foundational background will guide the development of angle concepts adapted to \(\bp\), grounded in the semi-inner product structure.

Let $(\mathcal{X},\|\cdot\|)$ be an arbitrary Banach space. The norm $\|\cdot\|$ is said to be Gâteaux differentiable at a nonzero element $x \in \mathcal{X}$ if  
\begin{equation}\label{gateaux eq}
\lim_{t\to 0^+}\frac{\|x+ty\|-\|x\|}{t}=\text{Re} D_x(y), \quad \text{for all } y\in\mathcal{X},
\end{equation}  
where $D_x$ is the unique functional in the dual space $\mathcal{X}^*$ satisfying  
\[
D_x(x)=\|x\| \quad \text{and} \quad \|D_x\|=1.
\]

In particular, when $\mathcal{X}=\mathcal{B}_p(\mathcal{H})$ with $1<p<\infty$, the space is uniformly convex and hence Gâteaux differentiable at every nonzero $X\in\mathcal{B}_p(\mathcal{H})$.  If $X=U|X|$ is the polar decomposition of $X$, then the corresponding functional $D_X\in\mathcal{B}_p(\mathcal{H})^*$  is given by
\[
D_X(Y)
=\|X\|_p^{\,1-p}\,\mathrm{tr}\bigl(|X|^{p-1}U^*Y\bigr)
\;=\;\frac{1}{\|X\|_p}\,[Y,X]_p,
\]
at every $ Y\in\mathcal{B}_p(\mathcal{H})$, with $[\cdot,\cdot]_p$ denotes the semi‐inner product defined in \eqref{semi inner prod def}.  For further details, see \cite{kittaneh_jmaa_1996}.  Therefore, the $\|\cdot\|_p$ norm is {G}\^{a}teaux diferenciable at every non-zero $X\in \bp$ and
\[\lim_{t\to 0^+}\frac{\|X+tY\|_p-\|X\|_p}{t}=\frac{1}{\|X\|_p}\text{Re} [Y,X]_p, \text{  for all } Y\in\bp.\]

Following the ideas of Mili\v{c}i\'{c} in \cite{M1, M2}, we define the function $g: \bp\times\bp\to \R$ by  
\begin{equation}\label{g-function}
g(X,Y)=\|X\|_p\text{Re}(D_X(Y))=\text{Re}[Y,X]_p.
\end{equation}  
This function serves as a fundamental tool to analyze geometric properties of $\bp$. In particular, it allows us to define the notion of angle between operators in this space.

The first notion of angle is given by  
\begin{equation}\label{angulo 1}
\angle(X,Y) = \arccos\left( \frac{\text{Re}\left([Y,X]_p+[X,Y]_p \right) }{2\|X\|_p\|Y\|_p} \right)  
= \arccos\left(\frac{\beta_{Y,X}+\beta_{X,Y}}{2} \right).
\end{equation}  
This expression generalizes the classical angle concept in Hilbert spaces by utilizing the semi-inner product $[\cdot,\cdot]_p$ instead of the standard inner product.

An alternative weighted angle also defined by  Mili\v{c}i\'{c} is: 
\begin{eqnarray}\label{angulo 2}
\angle_g(X,Y) &=& \arccos\left( \frac{\|X\|_p^2\text{Re}[Y,X]_p+\|Y\|_p^2\text{Re}[X,Y]_p}{\|X\|_p\|Y\|_p\left(\|X\|_p^2+\|Y\|_p^2 \right) } \right)  \nonumber \\
&=& \arccos\left(\frac{\|X\|_p^2\beta_{Y,X}+\|Y\|_p^2\beta_{X,Y}}{\|X\|_p^2+\|Y\|_p^2} \right).
\end{eqnarray}  
This formulation assigns different weights to the contributions of $X$ and $Y$, reflecting their respective norms. The denominator ensures proper normalization, maintaining consistency with the classical cosine function.

Analogously, following the ideas of Nur and Gunawan in \cite{NG},  we define the $gg$-angle as 
\begin{equation}\label{angulo gg}
		\angle_{gg}(X,Y)=\arccos\left( \dfrac{\sqrt{\left| \text{Re}[Y,X]_p\right| \left| \text{Re}[X,Y]_p\right| }}{\|X\|_p\|Y\|_p } \right)=
	\arccos\left( \sqrt{\left| \beta_{Y,X}\right| \left| \beta_{X,Y}\right| }\right).
\end{equation}

Various definitions of angles in the space $\bp$ have been proposed, each possessing unique properties that are not necessarily comparable with one another. Motivated by the diverse approaches arising from semi-inner product structures, we introduce a new definition of an angle in $\bp$. This definition resembles the one initially proposed by Mili\v{c}i\'{c} in \eqref{angulo 1}, which was formulated for a real normed space. However, since in our context the space of $p$-Schatten ideals is a complex normed space, we extend and adapt this notion to capture essential relationships in a coherent manner.

\begin{definition}
	Let $p>1$ and \(X, Y \in \bp\) with $X\neq 0$ and $Y\neq 0$, we define the $p$- angle between $X$ and $Y$ as follows
\begin{equation}\label{angulo BC1}
\angle_{p}(X,Y) =\arccos\left(\frac{|\alpha_{Y,X}|+|\alpha_{X,Y}|}{2} \right). 	
\end{equation}
\end{definition}
Observe that this is a well-defined angle since
\[
0\leq \frac{|\alpha_{Y,X}|+|\alpha_{X,Y}|}{2}\leq 1,
\]
where the last inequality follows from Proposition \ref{properties of alpha}.

We note that for any nonzero operators $X,Y\in\mathcal{B}_p(\mathcal{H})$, the angle $\angle_p(X,Y)$ satisfies the following fundamental properties:

\begin{theorem}
	Let $p>1$ and \(X, Y \in \bp\) with $X\neq 0$ and $Y\neq 0$. Then, the $p$-angle satisfies the  following conditions
	\begin{enumerate}
	\item \(0\leq \angle_p(X,Y) \leq \frac{\pi}{2}\).
	\item Symmetry, i.e. \(\angle_p(X,Y)=\angle_p(Y,X)\).
		\item \(\angle_p(aX,bY)=\angle_p(X,Y)\) for any nonzero scalars \(a, b \in \mathbb{C}\).
	\item \(\angle_p(X,Y)=\frac{\pi}{2}\) if and only if $X\obj Y$ and $Y\obj X$.

	\item If \(X \parallel^p Y\) (or equivalently \(Y \parallel^p X\)), then \(\angle_p(X,Y)=0\); in particular, if \(X\) and \(Y\) are linearly dependent, then \(\angle_p(X,Y)=0\).
	\item \(\angle_p(X,Y)\leq \angle_{gg}(X,Y)\).
\end{enumerate}
\end{theorem}

\begin{proof}
	
		Items (1), (2), and (3) follow directly from Proposition~\ref{properties of alpha}. Item (4) is a consequence of Corollary~\ref{coro equiv}, while item (5) follows from Theorem~\ref{|alpha|=1}. 
	
	We now present the proof of item (6). By applying the arithmetic-geometric mean inequality to $\sqrt{|\beta_{Y,X}|}$ and $\sqrt{|\beta_{X,Y}|}$, we obtain
	\[
	\sqrt{|\beta_{Y,X}| \cdot |\beta_{X,Y}|} \leq \frac{|\beta_{Y,X}| + |\beta_{X,Y}|}{2} \leq \frac{|\alpha_{Y,X}| + |\alpha_{X,Y}|}{2}.
	\]
	Since the function $\arccos$ is decreasing, it follows that
	\[
	\angle_p(X,Y) \leq \angle_{gg}(X,Y).
	\]
	
\end{proof}

These properties illustrate how \(\angle_p(X,Y)\) extends the concept of angles in \(\bp\) by capturing geometric relationships that are compatible with the structure imposed by the semi-inner product. In particular, the invariance under scalar multiplication and the comparison with the \(gg\)-angle emphasize its suitability for spectral and geometric analysis in \(\bp\).

\begin{theorem}
	Let $p>1$ and \(X, Y \in \bp\) with $X\neq 0$ and $Y\neq 0$, then 
	\begin{equation}
	\arccos\left(\sqrt{\frac{ \alpha_{|Y^*|,|X^*|}\alpha_{|Y|,|X|}+ \alpha_{|X^*|,|Y^*|}\alpha_{|X|,|Y|}}{2}}\right)\leq 	\angle_p(X,Y).
	\end{equation}
\end{theorem}
\begin{proof}
	Recall that the arithmetic-quadratic mean inequality states that, for any real numbers $x_1$ and $x_2$, 
	\[
	\frac{x_1+x_2}{2}\leq \sqrt{\frac{x_1^2+x_2^2}{2}}.
	\]
	Substituting $x_1$ and $x_2$ with $|\alpha_{Y,X}|$ and $|\alpha_{X,Y}|$, respectively, and using \eqref{cota CS alpha}, we obtain
	\[
	\frac{|\alpha_{Y,X}|+|\alpha_{X,Y}|}{2}\leq \sqrt{\frac{|\alpha_{Y,X}|^2+|\alpha_{X,Y}|^2}{2}}\leq \sqrt{\frac{\alpha_{|Y^*|,|X^*|}\alpha_{|Y|,|X|}+\alpha_{|X^*|,|Y^*|}\alpha_{|X|,|Y|}}{2}}.
	\]
	Finally, applying the $\arccos$ function, it follows that 
	\[
	\arccos\left(\frac{|\alpha_{Y,X}|+|\alpha_{X,Y}|}{2}\right) \geq \arccos\left(\sqrt{\frac{\alpha_{|Y^*|,|X^*|}\alpha_{|Y|,|X|}+\alpha_{|X^*|,|Y^*|}\alpha_{|X|,|Y|}}{2}}\right).
	\]
\end{proof}

Returning once more to the definitions \eqref{angulo 2} and \eqref{angulo gg}, note that it is possible to define analogous angles in $\bp$ by using $|\alpha_{Y,X}|$ instead of $|\beta_{Y,X}|$, as it appears in their formulation, namely:

\begin{equation}\label{angulo abis}
		\angle_{p,w}(X,Y)=\arccos\left(\frac{\|X\|_p^2|\alpha_{Y,X}|+\|Y\|_p^2|\alpha_{X,Y}|}{\|X\|_p^2+\|X\|_p^2} \right),
\end{equation}
\begin{equation}\label{angulo aa}
		\angle_{gg, p}(X,Y)=\arccos\left( \sqrt{\left| \alpha_{Y,X}\right| \left| \alpha_{X,Y}\right| }\right).
\end{equation}

We conclude this article by presenting a generalization of the $p$-angle previously introduced. This generalization is based on the notion of a mean. More precisely, recall that a function  $M : \mathbb{R}^n \to \mathbb{R}$ is called a mean if it satisfies some or all of the following properties:

		\begin{enumerate}
			\item Symmetry:
			\[
			M(x_1, \dots, x_n) = M(x_{\sigma(1)}, \dots, x_{\sigma(n)})
			\]
			for any permutation $\sigma$ of $\{1, \dots, n\}$.
			
			\item Monotonicity:
			If $x_i \leq y_i$ for all $i = 1, \dots, n$, then
			\[
			M(x_1, \dots, x_n) \leq M(y_1, \dots, y_n).
			\]
			
			\item Boundedness:
			\[
			\min\{x_1, \dots, x_n\} \leq M(x_1, \dots, x_n) \leq \max\{x_1, \dots, x_n\}.
			\]
			
			\item Idempotence:
			\[
			M(x, x, \dots, x) = x \quad \text{for all } x \in \mathbb{R}.
			\]
			
			\item Continuity: $M$ is a continuous function.
		\end{enumerate}
		
		Let $x_1, \dots, x_n \in \mathbb{R}_{\geq 0}$. Common examples of a mean  include:
		
		\begin{itemize}
			\item Arithmetic Mean (AM):
			\[
			\text{AM}(x_1, \dots, x_n) = \frac{1}{n} \sum_{i=1}^n x_i
			\]
			
			\item Geometric Mean (GM):
			\[
			\text{GM}(x_1, \dots, x_n) = \left( \prod_{i=1}^n x_i \right)^{1/n}
			\]
			
			\item Harmonic Mean (HM):
			\[
			\text{HM}(x_1, \dots, x_n) = \left( \frac{1}{n} \sum_{i=1}^n \frac{1}{x_i} \right)^{-1}, \quad \text{if } x_i \neq 0
			\]
			
			\item Quadratic Mean or Root Mean Square (QM):
			\[
			\text{QM}(x_1, \dots, x_n) = \sqrt{ \frac{1}{n} \sum_{i=1}^n x_i^2 }
			\]
		\end{itemize}

		For non-negative real numbers $x_1, \dots, x_n$, the following chain of inequalities holds:
		\[
		\text{HM} \leq \text{GM} \leq \text{AM} \leq \text{QM},
		\]
		with equality if and only if $x_1 = x_2 = \cdots = x_n$.
		
	Observe that the geometric mean of $|\beta_{Y,X}|$ and $|\beta_{X,Y}|$ satisfies
	\[
	\text{GM}(|\beta_{Y,X}|, |\beta_{X,Y}|) = \sqrt{|\beta_{Y,X}| |\beta_{X,Y}|},
	\]
	and the arithmetic mean of  $|\alpha_{Y,X}|$ and $|\alpha_{X,Y}|$ is given by
	\[
	\text{AM}(|\alpha_{Y,X}|, |\alpha_{X,Y}|) = \frac{|\alpha_{Y,X}| + |\alpha_{X,Y}|}{2}.
	\]
This motivates the definition of a family of angles in the space $\mathcal{B}_p(\mathcal{H})$ associated with any two-variable mean $M$, namely
\begin{equation}\label{angulopM}
\angle_{p, M}(X, Y)
= \arccos\bigl(M\bigl(|\alpha_{Y,X}|,\;|\alpha_{X,Y}|\bigr)\bigr),
\end{equation}
for nonzero $X,Y\in\mathcal{B}_p(\mathcal{H})$.  

In particular, one may take $M$ to be the Heinz mean: for $a,b\ge0$ and $t\in[0,1]$,
\[
H_t(a,b)=a^t b^{1-t}+a^{1-t}b^t.
\]
Note that $H_0(a,b)=H_1(a,b)=\mathrm{AM}(a,b)$ and $H_{1/2}(a,b)=\mathrm{GM}(a,b)$, and in general
\[
\mathrm{GM}(a,b)\;\le\;H_t(a,b)\;\le\;\mathrm{AM}(a,b),
\]
for all $t\in[0,1]$ \cite{bhatia}.  Hence an appealing choice is the Heinz-mean angle
\begin{equation}\label{angulopH}
\angle_{p, H}(X, Y)
= \arccos\!\bigl(|\alpha_{Y,X}|^t\,|\alpha_{X,Y}|^{1-t}
+|\alpha_{Y,X}|^{1-t}\,|\alpha_{X,Y}|^t\bigr).
\end{equation}

\end{document}